\documentclass{article}

\usepackage{bm}
\usepackage{graphicx,amssymb,amstext,amsmath,wrapfig,url}
\usepackage[hang]{subfigure}

\newcommand{\frakell}{\ensuremath{{\bm \ell}}}
\usepackage{jucs2e}
\newstytheorem{algorithm}{\bf}{\it}{Algorithm}[section]

\makeatletter

\newcommand{\out}[1]{}

\def\sectionD#1#2{
\section{#1}
\label{section:#2}
}

\def\subsectionD#1#2{
\subsection{#1}
\label{section:#2}
}

\def\paragraphD#1#2{
\paragraph*{#1}
\label{par:#2}
}

\newcommand{\lemlab}[1]{\label{lemma:#1}}
\newcommand{\theolab}[1]{\label{theo:#1}}

\newcommand{\alglab}[1]{\label{alg:#1}}

\newcommand{\proplab}[1]{\label{prop:#1}}

\newcommand{\lemref}[1]{Lemma \ref{lemma:#1}}
\newcommand{\theoref}[1]{Theorem \ref{theo:#1}}
\newcommand{\algref}[1]{Algorithm \ref{alg:#1}}

\newcommand{\figref}[1]{Figure \ref{fig:#1}}
\renewcommand{\eqref}[1]{(\ref{eq:#1})}
\newcommand{\secref}[1]{Section \ref{section:#1}}

\newcommand{\propref}[1]{Proposition \ref{prop:#1}}

\title{Graded Sparse Graphs and Matroids
\footnote{C.~S.~Calude, G.~Stefanescu, and M.~Zimand (eds.).
 {\em  Combinatorics and Related Areas. A Collection of Papers in Honour of the 65th Birthday of
Ioan Tomescu.}}}
\author{
{\bfseries Audrey Lee}\\
(University of Massachusetts, Amherst, MA\\ alee@cs.umass.edu)
\and
{\bfseries Ileana Streinu}\\
(Smith College, Northampton, MA \\ streinu@cs.smith.edu)
\and
{\bfseries Louis Theran}\\
(University of Massachusetts, Amherst, MA \\ theran@cs.umass.edu)
}

\out{
\title{Graded sparse graphs and matroids}
\author{{\bf Audrey Lee} \\ (University of Massachusetts, Amherst, MA \\ alee@cs.umass.edu) \and {\bf Ileana Streinu} \\ (Smith College, Northampton, MA \\ streinu@cs.smith.edu) \and 
{\bf Louis Theran} \\ (University of Massachusetts, Amherst, MA \\ theran@cs.umass.edu)}
}

\begin{document}
\maketitle

\begin{abstract}
Sparse graphs and their associated matroids play an important role in 
rigidity theory, where they capture the combinatorics 
of some families of generic minimally rigid structures.  We define a
new family called {\bf graded sparse graphs}, arising from 
generically pinned bar-and-joint frameworks, and prove that 
they also form matroids.
We also address several algorithmic problems on graded sparse graphs: {\bf Decision}, {\bf Spanning}, {\bf Extraction}, {\bf Components}, {\bf Optimization}, and {\bf Extension}.  We sketch variations on {\bf pebble game algorithms} to solve them.
\end{abstract}


\begin{category} F.2.2, G.2.2\end{category}

\sectionD{Introduction}{intro}
A {\bf bar-and-joint framework} is a planar structure made of fixed-length {\bf  bars}
connected by {\bf universal joints}.  Its allowed {\bf motions} are those that preserve the 
lengths and connectivity of the bars.  If the allowed motions are all {\bf trivial rigid
motions}, then the framework is {\bf rigid};  otherwise it is flexible.

Laman's foundational theorem \cite{laman}
characterizes generic minimally rigid bar-and-joint frameworks in terms of 
their underlying graph.  A Laman graph has $2n-3$ edges and the additional
property that every induced subgraph on $n'$ vertices spans at most $2n'-3$
edges.  Laman's theorem characterizes the graphs of generic minimally rigid
frameworks as Laman graphs.

Laman's hereditary counts have been recently generalized \cite{whiteley:Matroids:1996, pebblegame, hypergraphs} to { \em $(k,\ell)$-sparse graphs} and {\em hypergraphs}, which form the 
independent sets of a matroid called the {\bf $(k,\ell)$-sparsity matroid}.

\begin{figure}[htbp]
    \centering
    \subfigure[]{\includegraphics[width=0.4\textwidth]{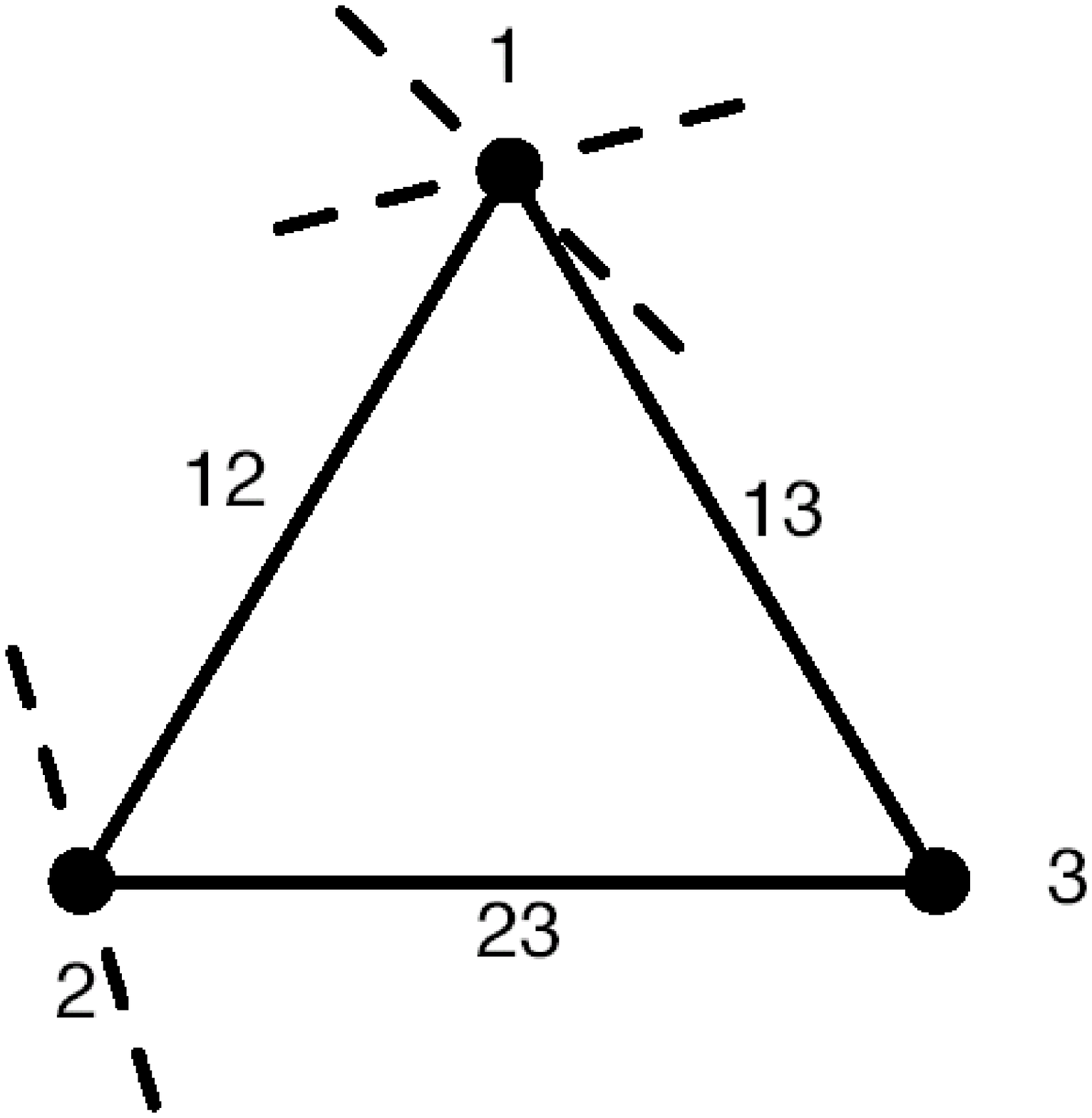}}
    \subfigure[]{\includegraphics[width=0.4\textwidth]{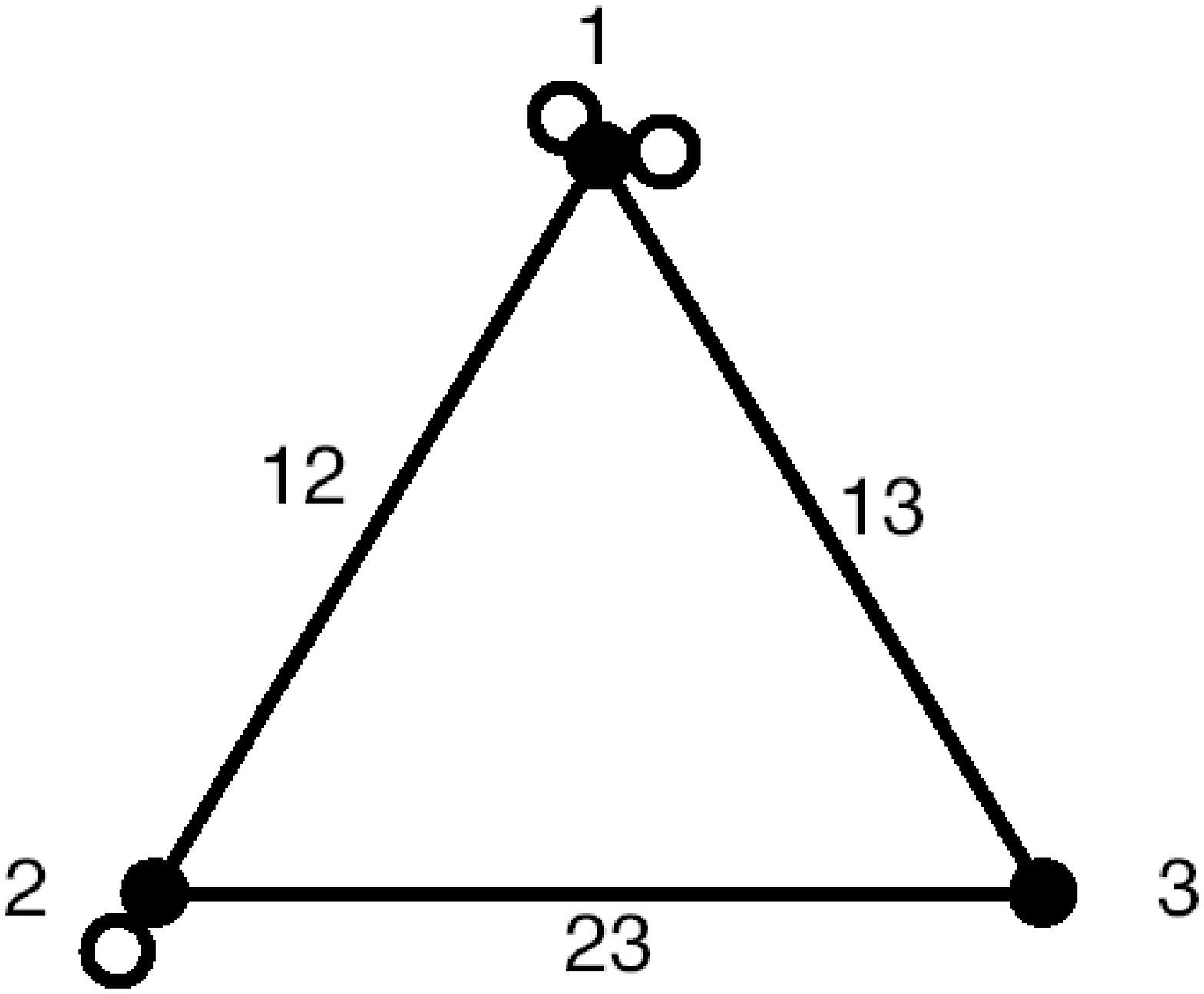}}
    \caption{Example of a bar-slider framework and its associated
graph: (a) a bar-slider framework; (b)
    the same framework given combinatorially as a graph with edges and loops}
    \label{fig:bar-slider-framework}
\end{figure}

In \cite{sliders,sliderscccg} 
we considered the problem of {\bf pinning} a 
bar-and-joint framework by adding {\bf sliders}.  
Pinning means completely immobilizing the structure by eliminating 
all the degrees of freedom, including the trivial rigid motions (rotations 
and translations).
We do this by constraining vertices to move along  generic lines, much like a
slider joint in mechanics.  A {\bf slider} at vertex $i$ is a line
$L_i$ associated with the vertex.  A structure made from bars, joints, and sliders is called 
a {\bf bar-slider} framework. 

We model a bar-slider framework combinatorially with a 
graph that has vertices for the joints, with {\bf edges} (2 endpoints) for the bars 
 and {\bf loops} (1 endpoint) for the sliders.  
\figref{bar-slider-framework} shows an example
of a bar-slider framework and its associated graph; the sliders are shown as
dotted lines running through a vertex.
The main rigidity result of \cite{sliders} is a Laman-type theorem.

\begin{proposition}[{\bf (Bar-slider framework rigidity)}]
	\proplab{slider-theorem}
Let $G$ be a graph with $2n-k$ edges and $k$ loops.  
$G$ is realizable as a generic minimally pinned bar-and-slider framework if 
and only if:
(1) Every subset of $n'$ vertices spans 
at most $2n'-3$ edges (not counting loops), and
(2) Every induced subgraph on $n'$ vertices spans at most $2n'$ edges and loops.
\end{proposition}

The generalization of the Laman counts from \propref{slider-theorem} leads to a 
{\bf pinning matroid}, which has as its bases the graphs of minimally pinned
generic bar-slider frameworks.

\paragraphD{Contributions.}{contrib} 
In this paper, we generalize the counts from \propref{slider-theorem} to what we call {\bf graded sparsity} on hypergraphs.  Graded sparsity has the same relationship to bar-slider structures as sparsity has to bar-and-joint frameworks. Complete definitions will be given in 
\secref{gradedmatroids}. We also briefly indicate the algorithmic solutions to the following fundamental problems (first posed in the context of pinning in \cite{sliders}), generalized to graded sparsity. 
\begin{description}
	\item[{\bf Decision problem:}] Is $G$ a  graded sparse graph?
    \item[{\bf Spanning problem:}] Does $G$ contain a spanning graded sparse graph?
	\item[{\bf Extraction problem:}] Find a maximum sized graded sparse subgraph
	of $G$.
	\item[{\bf Optimization problem:}]  Compute an optimal
	graded sparse subgraph of $G$ with respect to an arbitrary linear weight 
	function on the edges. 
	\item[{\bf Extension problem:}] Find a minimum size set of edges to add to $G$, so that it      becomes spanning. 
	\item [{\bf Components problem:}] Find the	{\it components} (which generalize rigid components to graded sparsity) of $G$.
\end{description}

For these problems, we give efficient, easily implementable algorithms
based on pebble games for general sparse graphs (see the papers
\cite{pebblegame}, \cite{hypergraphs}, \cite{components}).

\sectionD{Preliminaries}{prelim}
In this section, we give the necessary background (from previous work) to understand our contributions.  We start 
with sparse graphs and hypergraphs.

\subsectionD{Sparse graphs and hypergraphs.}{sparsity}
A {\bf hypergraph} $G=(V,E)$ is a finite set $V$ of $n$ vertices with a
set $E$ of $m$ edges that are subsets of $V$.  We allow multiple distinguished
copies of edges; i.e., our hypergraphs are multigraphs.  The {\bf dimension}
of an edge is the number of vertices in it; we call an edge of dimension $d$
a $d$-edge.  We call the vertices in an edge its {\bf endpoints}.  The 
concept of directed graphs extends to hypergraphs.  In a directed hypergraph,
each edge is given an {\bf orientation} ``away'' from a distinguished endpoint,
which we call its {\bf tail}.

A hypergraph is {\bf $(k,\ell)$-sparse} if every edge-induced subgraph with $m'$ edges
spanning $n'$ vertices satisfies $m'\le kn'-\ell$; a hypergraph that is $(k,\ell)$-sparse 
(shortly, sparse) and has $kn-\ell$ edges 
is called {\bf $(k,\ell)$-tight} (shortly, tight).  
Maximal tight subgraphs of a sparse hypergraph are called 
{\bf components}.

Sparse hypergraphs have a matroidal structure, first observed by White and Whiteley in the appendix of \cite{whiteley:Matroids:1996}. More specifically:

\begin{proposition}[{\bf (\cite{hypergraphs})}]\proplab{sparse-hypergraphs}
Let $G$ be a hypergraph on $n$ vertices.  For large enough $n$, the $(k,\ell)$-sparse
hypergraphs form the independent sets of a matroid that has tight hypergraphs as its bases.
\end{proposition}

When $\ell\ge dk$, all the edges in a sparse hypergraph
have dimension at least $d$, because otherwise the small edges
would violate sparsity and the matroid would be trivial. The $(k,\ell)$-sparsity matroid (for $0\leq \ell < dk$) is defined on the ground set  $K^+_n$, the complete hypergraph on $n$ vertices, where edges of dimension $d$ have multiplicity $dk$.  

\subsectionD{Pebble games}{pebblegames}
Pebble games are a family of simple construction rules for sparse hypergraphs.
For history and references, see \cite{pebblegame,hypergraphs}.
In a nutshell, the pebble game starts with an empty set
of vertices with $k$ pebbles on each vertex and proceeds through a sequence of
moves. Each move either adds a directed edge
or reorients one that is already there, using the location of the pebbles
on the graph to determine the allowed moves at each step.

Pebble games are indexed by non-negative integer parameters $k$
and $\ell$. 
Initially, every vertex starts with $k$ pebbles on
it.  An edge may be added if at least 
$\ell+1$ pebbles are present on its
endpoints, otherwise it is rejected.  When an edge is added, 
one of the pebbles is picked up from an endpoint 
and used to ``cover''  the new edge,
which is then directed away from that endpoint.  
Pebbles may be moved by reorienting edges.
If an endpoint of an edge, other than its tail, has 
at least one pebble, this pebble may be used to cover
the edge.  The edge is subsequently reoriented away from that endpoint, 
and the pebble previously covering the edge is returned 
to the original tail.

The pebble game is used as a basis for algorithms that solve
the fundamental sparse graph problems in \cite{pebblegame,hypergraphs}.  
The next proposition captures the results needed later.

\begin{proposition}[{\bf (\cite{hypergraphs})}]\proplab{pebble-game} 
{\bf Pebble games for sparse hypergraphs:}
Using the pebble game paradigm, the {\bf Decision} problem for sparse hypergraphs with edges of dimension $d$ can be solved in $O(dn^2)$ time and $O(n)$ space.  The {\bf Spanning}, {\bf Extraction} and {\bf Components} problems for hypergraphs with $m$ edges of dimension $d$ 
can be solved in $O(n^d)$ time and space or $O(nmd)$ time and $O(m)$ space.
{\bf Optimization} can be solved in either of these running times plus 
an additional $O(m\log m)$.
\end{proposition}

Not that in a hypergraph with edges of dimension $d$, $m$ may 
be $\Theta(n^d)$.

\subsectionD{Related work}{related}
Because of their relevance to rigidity, Laman graphs, and the
related families of sparse graphs, have been extensively studied. Classic papers include 
\cite{whiteley:unionMatroids:1988,whiteley:Matroids:1996}, where extensions to sparsity matroids on graphs and hypergraphs first appeared.  A more detailed history and comprehensive developments, as well as recent developments connecting sparse graphs and pebble game algorithms appear in \cite{pebblegame,hypergraphs}.
The graded matroids of this paper are a further generalization.

Another direction is encountered in  \cite{servatiusB:whiteley:constrainingCAD:1999}, which consider length and direction constraints (modeled as edges of different colors).  The associated rigidity counts require $(2,3)$-sparsity for monochromatic subgraphs and $(2,2)$-sparsity for bichromatic subgraphs. This extends sparsity  in a slightly different direction than we do here, and is neither a specialization nor a generalization of our graded sparsity.

Our algorithms fall into the family of generalized pebble games for
sparse hypergraphs \cite{pebblegame,hypergraphs,components}.  They are generalizations of an \cite{jacobs:hendrickson:PebbleGame:1997a}'s algorithm for Laman graphs, formally analyzed in \cite{berg:jordan:2003}.

\sectionD{Graded sparsity}{gradedmatroids}
In this section, we define the concept of graded sparsity 
and prove the main result.

\paragraphD{Definitions}{graded}
Let $G=(V,E)$ be a hypergraph.
A {\bf grading} $(E_1,E_2,\ldots,E_s)$ of $E$ 
is a strictly decreasing sequence of sets of edges 
$E=E_1\supsetneq E_2\supsetneq \cdots\supsetneq E_s$. An example is the {\bf standard grading}, where we fix the $E_i$'s to be edges of dimension at least $i$. This is the situation for the pinned Laman graphs of \cite{sliders}, where the grading consists of $\geq 1$-edges and $2$-edges.

Fix a grading on the edges of $G$. Define $G_{\ge i}$ as the subgraph of $G$ induced by
$\cup_{j\ge i}E_i$. 
Let $\frakell=\{\ell_1 < \ell_2 < \cdots < \ell_s \}$
be a vector of $s$ non-negative integers. We say that $G$ is 
{\bf $(k,\frakell)$-graded sparse} if
$G_{\ge i}$ is $(k,\ell_i)$-sparse for every $i$; $G$ is 
{\bf $(k,\frakell)$-graded tight}
if, in addition, it is $(k,\ell_1)$-tight.  To differentiate this concept from the
sparsity of \propref{sparse-hypergraphs}, we refer to $(k,\frakell)$-graded 
sparsity as {\bf graded sparsity}.  The {\bf components} of a graded sparse graph
$G$ are the $(k,\ell_1)$-components of $G$.

\paragraphD{Main result}{mainresult}
It can be easily shown that the family of $(k,\frakell)$-graded sparse graphs is the intersection of $s$
matroidal families of graphs. The main result of this paper is the following stronger property.

\begin{theorem}[{\bf (Graded sparsity matroids)}]\theolab{graded-matroid}
${}$\\
The $(k,\frakell)$-graded sparse hypergraphs form the independent sets of a matroid. For
 large enough $n$, the $(k,\frakell)$-graded tight hypergraphs are its bases.
\end{theorem}

The proof of \theoref{graded-matroid} is based on the circuit axioms 
for matroids.  See \cite{oxley:matroid} for an introduction to 
matroid theory.  We start by
formulating $(k,\frakell)$-graded sparsity in terms of circuits.  

For a $(k,\ell)$-sparsity matroid, the $(k,\ell)$-circuits are exactly the graphs on $n'$
vertices with $kn'-\ell+1$ edges such that every proper subgraph is $(k,\ell)$-sparse.

We now recursively define a family $\mathcal{C}$ as follows: $\mathcal{C}_{\ge s}$
is the set of $(k,\frakell_s)$-circuits in $G_{ s}$; for $i<s$, $\mathcal{C}_{\ge i}$
is the union of $\mathcal{C}_{\ge i+1}$ and the $(k,\frakell_i)$-circuits of $G_{\ge i}$
that do not contain any of the elements of $\mathcal{C}_{\ge i+1}$.  Finally, set
$\mathcal{C}=\mathcal{C}_{\ge 1}$.

{\bf Example:} As an example of the construction of $\mathcal{C}$, we consider the 
case of $k=1$, $\frakell=(0,1)$ with the standard
grading.  $\mathcal{C}_{\ge 2}$ consists of the $(1,1)$-circuits of edges; a fundamental
result of graph theory \cite{Na61,tutte:decomposing-graph-in-factors-1961} says that
these are the simple cycles of edges.  Using the identification of $(1,0)$-tight 
graphs with graphs having exactly one cycle per connected component (see \cite{maps}
for details and references), we infer that the $(1,0)$-circuits are pairs of cycles sharing 
edges or connected by a path.  Since cycles involving edges are already in $\mathcal{C}_{\ge 2}$,
$\mathcal{C}_{\ge 1}$ adds only pairs of loops connected by a simple path.

\medskip

We now prove a structural property of $\mathcal{C}$ that 
relates $\mathcal{C}$ to $(k,\frakell)$-graded
 sparsity and  will be used in the proof
of \theoref{graded-matroid}.

\begin{lemma}\lemlab{circuit-dimension}
Let $d$, $k$, and $\ell_i$ be such that $(d-1)k\le \ell_i<dk$. 
Then, every set in $\mathcal{C}_{\ge i}$
is either a single edge or has only edges of dimension at least $d$.
\end{lemma}
\begin{proof}
A structure theorem from \cite{hypergraphs} says that for $k$ and $\ell_i$
satisfying the condition in the lemma, all sparse graphs have only edges of dimension
at least $d$ or are empty.
Since any proper subgraph of a $(k,a)$-circuit for $a\ge \ell_i$
is $(k,\ell_i)$-sparse, either the circuit has only edges of dimension at
least $d$ or only empty proper subgraphs, i.e. it has exactly one edge.
\end{proof}

\begin{lemma}\lemlab{circuit-sparsity}
A hypergraph $G$ is $(k,\frakell)$-graded sparse if and only if it does not
contain a subgraph in $\mathcal{C}$.
\end{lemma}
\begin{proof}
It is clear that all the $(k,\frakell)$-graded sparse hypergraphs avoid the
subgraphs in $\mathcal{C}$, since they cannot contain {\it any}
$(k,\ell_i)$-circuit of $G_{\ge i}$.

For the reverse inclusion, suppose that $G$ is not sparse.  Then for
some $i$, $G_{\ge i}$ is not $(k,\ell_i)$-sparse.  This is
equivalent to saying that $G_{\ge i}$ contains some $(k,\ell_i)$-circuit $C$.
There are now two cases: if $C\in \mathcal{C}$ we are done; if
not, then some $C'\subsetneq C$ is in $\mathcal{C}$,
and $G$ contains $C'$, which completes the proof.
\end{proof}

We are now ready to prove \theoref{graded-matroid}.
\begin{proof}[of \theoref{graded-matroid}]
\lemref{circuit-sparsity} says that it is sufficient to verify that $\mathcal{C}$
obeys the circuit axioms.  By construction, $\mathcal{C}$ does not contain the empty
set and no sets of $\mathcal{C}$ contain each other.

What is left to prove is that, 
for $C_i, C_j\in \mathcal{C}$ with $y\in C_i\cap C_j$,
$(C_i\cup C_j)-y$ contains an element of $\mathcal{C}$.
Suppose that $C_i$ is a $(k,\ell_i)$-circuit
and $C_j$ is a $(k,\ell_j)$-circuit with $j\ge i$.  Let $m_i, m_j,m_\cup$ and $m_\cap$ be the
number of edges in $C_i$, $C_j$, $C_i\cup C_j$, and $C_i\cap C_j$, respectively.
Similarly define $n_i, n_j, n_\cup$, and $n_\cap$ for the size of the vertex sets.

\lemref{circuit-dimension} implies that $y$ has dimension $d$, where $\ell_j<dk$;
otherwise $C_j$ would have to be the single edge $y$, and this would block $C_i$'s inclusion in
$\mathcal{C}$ (since $j\geq i$).  Because $C_i\cap C_j\subsetneq C_j$, we have $n_\cap\ge d$ and
$m_\cap\le kn_\cap-\frakell_j$.  By counting edges, we have
\begin{align*}
    m_\cup & = m_i+m_j-m_\cap\ge m_i+m_j-(kn_\cap-\ell_j) \\
    &= kn_i-\ell_i+1+kn_j-\ell_j + 1 -(kn_\cap-\ell_j)  \\
    & = kn_\cup-\ell_i+2.
\end{align*}
It follows that $C_i\cup C_j$ cannot be $(k,\ell_i)$-sparse, and by \lemref{circuit-sparsity},
this is equivalent to having an element of $\mathcal{C}$ as a subgraph.
\end{proof}

\sectionD{Algorithms}{algorithms-graded}
In this section, we start with the algorithm for {\bf Extraction}  
and then derive algorithms for the other problems from it. 

\begin{algorithm}{\bf Extraction of a graded sparse graph} \alglab{graded}

\noindent {\bf Input:} A hypergraph $G$, with a grading on the edges. \\
\noindent {\bf Output:} A maximum size $(k,\frakell)$-graded sparse subgraph of $G$.

\noindent {\bf Method:} Initialize the pebble game with $k$ 
pebbles on every vertex.

Iterate over the edges of $G$ in an arbitrary order. 
For each edge $e$:
\begin{enumerate}
    \item Try to collect at least 
    $\ell_1+1$ pebbles on the endpoints of $e$.
    If this is not possible, reject it.
    \item If $e\in E_1$, accept it, using the rules of the $(k,\ell_1)$-pebble game.
    \item Otherwise, copy the configuration of the pebble game into a ``shadow graph''.  Set $d\leftarrow 1$.
    \item In the
    shadow, remove every edge of $E_d$, and put a pebble on the tail 
 	of the removed edges. 
	\item Try to collect $\ell_{d+1}-\ell_d$ more pebbles on the endpoints of $e$.
	There are three possibilities: (1) if the pebbles cannot be collected, reject $e$
	and discard the shadow; (2) otherwise there are $\ge \ell_{d+1}+1$ pebbles on the
	endpoints of $e$, if $e\in E_{d+1}$ discard the shadow and accept $e$ using the rules of the
	$(k,\ell_{d+1})$ pebble game; (3) otherwise, if $e \not\in E_{d+1}$, set $d\leftarrow d+1$
	and go to step 4.
\end{enumerate}

Finally, output all the accepted edges.
\end{algorithm}

{\bf Correctness:} By \theoref{graded-matroid},  adding edges in
any order leads to a sparse graph of maximum size.  What is left to check is that the edges
accepted are exactly the independent ones.

This follows from the fact that moving pebbles in the pebble
game is a reversible process, except when a pebble is moved and then used
to cover a new edge.  Since the pebbles covering hyper-edges in $E_j$, for $j<i$,
would be on the vertices where they are located in the $(k,\frakell_i)$-pebble game,
for $j<i$, then \algref{graded} accepts
edges in $E_i$ exactly when the $(k,\frakell_i)$-pebble game would.  By results
of \cite{pebblegame,hypergraphs}, we conclude that \algref{graded} computes
a maximum size $(k,\frakell)$-graded sparse subgraph of $G$.

{\bf Running time:}
If we maintain components,
the running time is $O(n^{d^{\star}})$, where $d^\star$
is the dimension of the largest hyperedge in the input. Without maintaining components, 
the running time is $O(mnd)$, with the caveat that $m$ can be $\Theta(n^d)$.

\paragraphD{\rm\bf Application to the fundamental problems.}{propblems}
We can use \algref{graded} to solve the remaining fundamental problems. 
For {\bf Decison} a simplification yields a running time of $O(dn^2)$:
checking for the correct number of edges is an $O(n)$ step, and after that
$O(mnd)$ becomes $O(n^2d)$.

For {\bf Optimization}, we
consider the edges in an order sorted by weight.  The correctness of
this approach follows from the characterization of matroids by greedy
algorithms.  Because of the sorting phase, the running time is $O(n^d+m\log m)$.

The components are the $(k,\frakell_1)$-components of the output.  Since
we maintain these anyway, the running time for {\bf Components} is $O(n^d)$.

Finally, for {\bf Extension}, the matroidal property of $(k,\frakell)$-graded sparse graphs implies that it can be solved by using the {\bf Extraction} algorithm on $K^+_n$,
considering the edges of the given independent set first and the rest of $E(K^+_n)$
in any desired order. The solution to {\bf Spanning} is similar. 

{\bf Acknowledgement.}
This research was supported by the NSF CCF 0728783 grant of Ileana Streinu.


\begin{thebibliography}{99}

\bibitem[Berg and Jord{\'a}n 2003]{berg:jordan:2003}
Berg, A.~R., Jord{\'a}n, T.:
``Algorithms for graph rigidity and scene analysis'';
Proc. 11th ESA, LNCS volume 2832, Springer (2003), 78-89.

\bibitem[Haas 2002]{haas:2002}
Haas, R.:
``Characterizations of arboricity of graphs'';
Ars Combinatorica 63 (2002), 129-137.

\bibitem[Haas et~al. 2007]{maps}
Haas, R., Lee, A., Streinu, I.,  Theran, L.:
``Characterizing sparse graphs by map decompositions'';
Journal of Combinatorial Mathematics and Combinatorial Computing, vol. 62, 3--11 (2007);
\url{http://arxiv.org/abs/0704.3843}.

\bibitem[Hendrickson 1992]{hendrickson:uniqueRealizabili%
ty:1992}
Hendrickson, B.:
``Conditions for unique graph realizations'';
SIAM Journal of Computing 21, 1 (1992), 65-84.

\bibitem[Jacobs and  Hendrickson 1997]{jacobs:hendrickson:PebbleGame:1997a}
Jacobs, D.~J., Hendrickson, B.:
``An algorithm for two-dimensional rigidity percolation: the pebble
  game'';
Journal of Computational Physics 137 (1992), 346-365.

\bibitem[Laman 1970]{laman}
Laman, G.:
``On graphs and rigidity of plane skeletal structures'';
Journal of Engineering Mathematics 4 (1970), 331-340.

\bibitem[Lee and Streinu 2007]{pebblegame}
Lee, A., Streinu, I.:
``Pebble game algorithms and sparse graphs'';
Discrete Mathematics, to appear;
\url{http://arxiv.org/abs/math/0702129}.

\bibitem[Lee et~al. 2005]{components}
Lee, A., Streinu, I., Theran, L.:
``Finding and maintaining rigid components'';
Proceedings of the 17th Canadian Conference on Computational
  Geometry, Windsor, Ontario (2005);
\url{http://cccg.cs.uwindsor.ca/papers/72.pdf}.

\bibitem[Lee et~al. 2007]{sliderscccg}
Lee, A., Streinu, I., Theran, L.:
``The slider-pinning problem'';
Proceedings of the 19th Canadian Conference on Computational
  Geometry, Ottawa, Ontario (2007);
\url{http://linkage.cs.umass.edu/publications/sliders-cccg.pdf}.

\bibitem[Nash-Williams 1961]{Na61}
Nash-Williams, C. S. J.~A.:
``Edge-disjoint spanning trees of finite graphs'';
Journal London Math. Soc. 36 (1961), 445-450.

\bibitem[Oxley 1992]{oxley:matroid}
Oxley, J.~G.:
``Matroid theory'';
The Clarendon Press Oxford University Press, New York (1961).

\bibitem[Servatius and Whiteley 1999]{servatiusB:whiteley:constrainingCAD:1999}
Servatius, B., Whiteley, W.:
``Constraining plane configurations in {CAD}: Combinatorics of directions and lengths'';
SIAM Journal on Discrete Mathematics 12(1) (1999), 136-153.

\bibitem[Streinu and Theran 2007]{hypergraphs}
Streinu, I., Theran, L.:
``Sparse hypergraphs and pebble game algorithms''
European Journal of Combinatorics, special issue for Oriented Matroids'05, to appear (2008);
\url{http://arxiv.org/abs/math/0703921}.

\bibitem[Streinu and Theran 2007b]{sliders}
Streinu, I., Theran, L.:
``Combinatorial genericity and minimal rigidity'';
Manuscript (2007); \url{http://arxiv.org/abs/0712.0031}.

\bibitem[Tay 1984]{Ta2}
Tay, T.-S.:
``Rigidity of multigraphs {I}: linking rigid bodies in $n$-space'';
Journal of Combinatorial Theory Series B 26 (1994), 95-112.

\bibitem[Tutte 1961]{tutte:decomposing-graph-in-factors-%
1961}
Tutte, W.~T.:
``On the problem of decomposing a graph into $n$ connected factors'';
Journal London Math. Soc. 142 (1961), 221-230.

\bibitem[Whiteley 1988]{whiteley:unionMatroids:1988}
Whiteley, W.:
``The union of matroids and the rigidity of frameworks'';
SIAM Journal Discrete Mathematics 1, 2 (1988), 237-255.

\bibitem[Whiteley 1996]{whiteley:Matroids:1996}
Whiteley, W.:
``Some matroids from discrete applied geometry'';
In: Matroid Theory (Bonin, J., Oxley, J.~G., Servatius, B., Editors),
  volume 197 of {\em Contemporary Mathematics}, American Mathematical
  Society (1996), 171-311.

\end{thebibliography}
\end{document}